\documentclass[11pt]{article}
\usepackage{amssymb}
\usepackage[usenames,dvipsnames]{xcolor}
\usepackage{amsthm}
\usepackage{amsmath}
\usepackage{amsfonts}

\usepackage{graphicx}

 \theoremstyle{definition}

 \numberwithin{equation}{section}

\newtheorem{theorem}{Theorem}[section]
\newtheorem{lemma}[theorem]{Lemma}

\theoremstyle{definition}

\theoremstyle{remark}
\newtheorem{remark}[theorem]{Remark}

\begin{document}

\title{Uniform Bound of Sobolev Norms of Solutions to 3D Nonlinear
Wave Equations with Null Condition}
\author{Fan Wang\footnote{School of
Mathematical Sciences, Fudan University, Shanghai 200433, P.R.China.
{\it Email: wangf767@gmail.com}}}
\date{}
\maketitle

\begin{abstract}
This article concerns the time growth of Sobolev norms of classical
solutions to the 3D quasi-linear wave equations with the null
condition. Given initial data in $H^s\times H^{s-1}$ with compact supports, the global
well-posedness theory has been established independently
by Klainerman \cite{Klainerman_1986} and Christodoulou \cite{Christodoulou_1986},
respectively, for a relatively large integer $s$. However,  the highest order Sobolev
energy, namely, the $H^s$ energy of solutions may have a logrithmatic growth in time. In this paper, we show that the $H^s$ energy of solutions
is also uniformly bounded for $s \geq 5$. The proof employs the generalized energy method
of Klainerman, enhanced by weighted $L^2$ estimates and the ghost weight introduced by Alinhac.
\end{abstract}

\section{Introduction}
In this paper we study the time growth rate of the highest Sobolev norm of
solutions to the nonlinear wave equations of the form
\begin{align}
\Box u=B_{\lambda\mu\nu}\partial_\lambda u \partial_\mu\partial_\nu u,\label{Equation}\\
u(0, \cdot)=\varphi,\ \ \ \ \partial_tu(0, \cdot)=\psi,\label{data}
\end{align}
in three spatial dimensions. Here the Greek indices $\lambda, \mu, \nu$ range from 0 to 3, in this paper,
Latin induces range from 1 to 3, and the Einstein convention is applied over repeated indices.
The D'Alembertian, which acts on scalar-valued
functions $u$: $[0,T)\times \mathbb{R}^3\rightarrow \mathbb{R}$, is denoted by
$\Box =\partial^2_{tt}-\Delta$. $B_{\lambda\mu\nu}$ are constants which satisfies the following symmetry:
\begin{equation}
B_{\lambda\mu\nu}=B_{\lambda\nu\mu}.
\label{symmetric}
\end{equation}
For the initial data $\varphi$, $\psi$,
we assume that
\begin{equation}
(\varphi, \psi) \in H^{s}_0(B_R)\times H^{s-1}_0(B_R),\quad s \geq 5, \label{Data}
\end{equation}
where $B_R$ is a ball of radius $R$ for some $R > 0$, which centers at the origin.

We denote the bilinear form coming from nonlinearities by
\begin{equation}
N(u,v)=B_{\lambda\mu\nu}\partial_\lambda u \partial_\mu\partial_\nu v.
\label{Nonlinearity}
\end{equation}
Suggested by S. Klainerman (\cite{Klainerman_1982}), the nonlinearities in \eqref{Equation} are called to satisfy the so-called null condition
if
\begin{equation}
B_{\lambda\mu\nu} X_\lambda X_\mu X_\nu=0,
\label{Null}
\end{equation}
for all $X \in \mathcal{N}$, where $\mathcal{N}$ is the hypersurface defined as
$$\mathcal{N}=\{X \in \mathbb{R}^{4}: X_0^2-(X_1^2+X_2^2+X_3^2)=0\}.$$
 In this paper we set $X_0=-1$ without loss of generality.

To place our result in context, we review a few highlights from the existence theory of nonlinear wave equations.
The global existence of the solutions to wave equation in two and three dimensions have
been studied by many authors and we will not attempt to exhaust references in this paper. The initial value problem for small solutions of 3D
quasilinear wave equations with quadratic nonlinearities is almost globally well-posed \cite{F. John_1984}, and
in general this is sharp \cite{F. John_1981}.
The global well-posedness of classical solutions to the Cauchy problem for the quasi-linear wave equations \eqref{Equation}-\eqref{data} was established
in 1986 independently by Klainerman
\cite{Klainerman_1986} and Christodoulou \cite{Christodoulou_1986}, under the assumption that the nonlinearities satisfy the null condition \eqref{Null}, for sufficiently small initial data. More precisely, for initial data in $H^{s}_0(B_R)\times H^{s-1}_0(B_R)$, here $s$ is a relatively large integer, the result in \cite{Klainerman_1986} implies that the (lower order) $H^{s-2}$ energy of solutions is uniformly bounded, while the highest Sobolev energy, namely, $H^s$ energy may have a logrithmatic
growth in time. The highest Sobolev energy of solutions of Christodoulou in \cite{Christodoulou_1986} also has a kind of time growth. Here and in what follows, by $H^s$ energy of a solution $u(t, x)$ we mean
\begin{equation}
\mathcal{H}^s(u(t)) = \sup_{t > 0}\|\partial_tu(t, \cdot)\|_{H^{s-1}}^2 + \sup_{t > 0}\|\nabla u(t, \cdot)\|_{H^{s-1}}^2.
\end{equation}

In this paper, we prove that the $H^s$ energy of solutions is also uniformly
bounded starting from a $H^{s}_0(B_R)\times H^{s-1}_0(B_R)$ data. Our main result is as follows:

\begin{theorem}\label{Main}
Let $s\geq 5$. Assume that the nonlinear terms in (\ref{Equation}) satisfy
the symmetry condition (\ref{symmetric}) and null condition (\ref{Null}).
Then the initial value problem for (\ref{Equation}) with initial data (\ref{data})
satisfying
$$E_{s}(u(0))<\epsilon$$
is globally well-posed for a sufficiently small $\epsilon > 0$  and the highest order
energy is globally bounded: $E_s(u(t)) \leq C\epsilon$.
\end{theorem}
The definition of the generalized energy $E_{s}(u(t))$ is given in \eqref{GE}.

Before ending the introduction, let us describe technical features in the present paper at length.
The result of Christodoulou\cite{Christodoulou_1986} relied on the
conformal method, the null condition implies then that the nonlinear terms of
the equations transform into smooth terms, and the problem is reduced to a
local problem with small data. While the proof of Klainerman uses a special energy
inequality for the wave equation, which is obtained by multiplying by an appropriate vector field
with quadratic coefficients (see \cite{Hormander_1997} for an account of both aspects).
They obtained the uniform bound of
the $H^{s - 2}$ energy norm of solutions if the initial data $(\phi, \psi) \in H^{s} \times H^{s-1}$
with compact support, but the $H^s$ energy norm of the solutions may have a slightly
power function growth in time. For elastic waves, in \cite{Sideris_1996}, T. C. Sideris carried out the energy integral argument by deriving a pair of coupled differential inequalities
for a higher order energy and a lower order energy,
 which was used
extensively in later studies
\cite{K. Hidano_2004, K. Hidano_20041, A. Hoshiga_2000, S. Katayama_1993, Sideris_2000, Sideris_2002, K. Yokoyama_2000}
by a lot of authors. The authors distinguished two different levels energy,
the lower order Sobolev energy of which must remain small, while the highest Sobolev energy will
grow polynomially in time. We notice that in the estimation for the highest order
energy, the null condition plays no role and hence the highest order
energy may have a slightly power function growth in time. The null condition
only comes in when estimating the lower order energy, which yields an extra
time decay and a uniform bound for Sobolev norm in $H^{s - 2}$.

In this paper, we show that for 3D scalar quasilinear wave equations with null condition the $H^s$ energy of solutions
is also uniformly bounded for $s \geq 5$. Our analysis is in line with
that of Zhen Lei, T. C. Sideris, Yi Zhou
\cite{LZT_2012},  The proof employs the generalized energy method
of Klainerman, enhanced by weighted $L^2$ estimates and the ghost weight introduced by Alinhac \cite{S. Alinhac_2001}, which was used
to study the global well-posedness of
classical solutions to nonlinear wave equations with null condition in 2D.
The advantage of \cite{LZT_2012} is that one can still use the null condition in the highest
order energy estimate, which is why we can get a uniform bound for the highest
order energy. At the time of this writing, it is still open whether the highest
order energy is uniformly bounded or not in the case of 2D nonlinear wave
equations. The corresponding problem is also open for nonlinear wave systems
with different speeds in both two and three dimensions.
In addition, we wish push forward our method to proof the global existence of the 2D incompressible
isotropic elastodynamics in the future. After we completed this paper, we learn from \cite{S. Alinhac_2010}
page 94, in which a similar results were obtained by Alinhac.

This paper is organized as follows: In section 2 we will introduce some notations
and recall some basic tools which are necessary for the  proof of our Theorem.
Then in section 3 and section 4 we present the proof of Theorem 1.1. In this paper,
the letter C denotes various positive and finite constants whose exact values
are unimportant and may vary from line to line.

\section{Preliminaries}
We explain the notation used in this paper, points in $\mathbb{R}^4$ will be denoted by $X=(x_0, x_1, x_2, x_3)=(t,x)$, and $r=|x|$. Partial derivatives will be written as $\partial_k=\frac{\partial}{\partial{x_k}}.$
Following S. Klainerman \cite{Klainerman_1985}, introduce a set of partial differential operators
$$\partial=(\partial_t, \nabla)=(\partial_t, \partial_1, \partial_2, \partial_3),$$
$$\Omega=(\Omega_1, \Omega_2, \Omega_3)=x\wedge\nabla=
(x_2\partial_3-x_3\partial_2, x_3\partial_1-x_1\partial_3, x_1\partial_2-x_2\partial_1),$$
$$L=(L_1, L_2, L_3)=(t\partial_1+x_1\partial_t, t\partial_2+x_2\partial_t, t\partial_3+x_3\partial_t),$$
$$S=t\partial_t+x\cdot\nabla=t\partial_t+r\partial_r.$$
The eleven vector fields will be written as
$\Gamma=(\Gamma_0, \cdots, \Gamma_{10})=(\partial, \Omega, L, S).$
Thus, the spatial derivatives can be decomposed as
\begin{equation}
\nabla=\frac{x}{r}\partial_r-\frac{x}{r^2}\wedge\Omega.
\label{Decompose}
\end{equation}
For multi-indices $a=(a_0, \cdots, a_{10})$, we denote
$$\Gamma^a=\Gamma_0^{a_0}\cdots \Gamma_{10}^{a_{10}}.$$
It is convenient to set $\Gamma^a=1$ if $|a|=0$. Suppose that $b$ and $c$ are
disjoint subsequences of $a$. Then we will say $b+c=a$, if $ b_i+c_i =a_i$ for
all $i$, and $b+c\leq a$, if $b_i + c_i \leq a_i$ for all $i$.

Associated with the D'Alembertian operator, in this paper, we define a new energy
\begin{equation}
E_1(u(t))=\frac{1}{2}\int_{\mathbb{R}^3}(|\partial_tu|^2e^{-q}+|\nabla u|^2e^{-q})dx,
\end{equation}
where $q'=\frac{1}{1+\sigma^2}$, with $q(0) = 0$, $\sigma=t-r$. From which
we know that $q$ is uniformly bounded from below and above.
The generalized energy is defined as
\begin{equation}\label{GE}
E_s(u(t))=\sum_{|a|\leq s-1} E_1(\Gamma^a u(t)),      s=2, 3, \cdots
\end{equation}

In order to proof theorem \ref{Main}, we shall record the following several lemmas.
The first one just describes the relationship between the usual derivative $\partial$ and the $\Gamma$ operator.
\begin{lemma}\label{lemma1} For any integer $N\geq0$, we have
\begin{equation}
\sum_{|\alpha|= N}|\partial^\alpha u(t, x)|\leq C_N(1+|t-|x||)^{-N}\sum_{|\beta|\leq N}|\Gamma^\beta u(t, x)|,
\label{D}
\end{equation}
where $\alpha$ and $\beta$ are multi-indices and $C_N$ is a positive constant depending on $N$.
\end{lemma}
\begin{proof}
See \cite{LZ_1995}.
\end{proof}
We also list a special case for $|\alpha|= 1$, which will be constantly used in this paper.
\begin{remark}
Suppose $|\alpha|=1,$ we have
\begin{equation}
|\partial u(t,x)\leq C(1+|t-x|)^{-1}\sum_{|\beta|=1}|\Gamma^\beta u(t,x)|.
\end{equation}
\end{remark}
We shall also need the following Sobolev-type inequalities with weight. It's proof appeared
implicity in \cite{KlainermanT_1996}.
\begin{lemma}\label{lemma2}
For any $u(t, x)\in\mathcal{ C}_0^\infty(\mathbb{R}^3)$, we have
\begin{equation}
r^{\frac{1}{2}}|u(t, x)|\leq C\sum_{|\beta|\leq1}||\nabla\Omega^\beta u(t, x)||_{L^2}
\end{equation}
provided that the norm on the right hand side is finite.
\end{lemma}

The next lemma is Klainerman's inequalities which will be constantly used in the energy estimates.
\begin{lemma}\label{lemma3} Suppose that $u=u(t,x)$ is a function with compact
support in the variable $x$ for any fixed $t\geq 0$. Then for any integer $N\geq0$, we have
\begin{equation}
\begin{split}
&\ \ \ \ \sum_{|\alpha|\leq N}|\Gamma^\alpha u(t, .)|_{L^\infty}\\
&\leq C(1+t+|x|)^{-\frac{n-1}{p}}(1+|t-x|)^{-\frac{1}{p}}
\sum_{|\alpha|\leq N+[\frac{n}{p}]+1}\|\Gamma^\alpha u(t, .)\|_{L^p},\ \ \ \ \forall t\geq0,
\end{split}
\label{K}
\end{equation}
where $p\geq1$ and $C$ is a positive constant.
\end{lemma}
\begin{proof}
See S. Klainerman \cite{Klainerman_1987}.
\end{proof}
\begin{remark}
From lemma \ref{lemma3}, we obtain
$|u(t, x)|\leq C(1+t)^{-\frac{n}{2}}\|u(t, .)\|_{\Gamma, s, 2}$, $s>\frac{n}{2}$, for $|x|\leq\frac{t+1}{2}$. And $|u(t, x)|\leq C(1+t)^{-\frac{n-1}{2}}\|u(t, .)\|_{\Gamma, s, 2}$, $s>\frac{n}{2}$, for $|x|\geq\frac{t+1}{2}$, here
$\|u(t, .)\|_{\Gamma, s, 2}\triangleq \sum_{|\alpha|\leq s}\|\Gamma^\alpha u\|_{L^2}$ and $s$ is a integer.
\end{remark}
The following lemma imply that the operator $\partial_t+\partial_r$ has a $(1+t)^{-1}$ decay in the region $|x|\geq\frac{t+1}{2}$.
\begin{lemma}\label{lemma4}
Suppose $|x|\geq\frac{t+1}{2}$, then we have $|(\partial_t+\partial_r)u|\leq \frac{C}{1+t}|\Gamma u|$.
\end{lemma}
\begin{proof}
Recalling the definitions of $L_0$ and $L_i$,
\begin{equation}
L_0=t\partial_t+x_1\partial_{x_1}+\cdots+x_n\partial_{x_n}=t\partial_t+r\partial_r,
\label{L_0}
\end{equation}
\begin{equation}
L_i=t\partial_{x_i}+x_i\partial_t.
\label{L_1}
\end{equation}
Multiplying (\ref{L_1}) by $x_i$,
\begin{equation}
x_iL_i=tx_i\partial_{x_i}+x^2_i\partial_t=tr\partial_r+r^2\partial_t.
\label{L_x}
\end{equation}
Dividing (\ref{L_x}) by $r$ and adding (\ref{L_0}), we obtain the following identity\\
\begin{equation}
\partial_t+\partial_r=\frac{L_0+\frac{x_i}{r}L_i}{t+r}
\end{equation}
This completes the proof.
\end{proof}
The next result captures the manner in which the null condition will be useful in the course of the energy estimates.
\begin{lemma}\label{N-inequality}
Suppose that the nonlinear form $N(u, u)$ satisfies the null condition (\ref{Null}). For $r\geq\frac{t+1}{2}$, we have
\begin{equation}
|B_{\lambda\mu\nu}\partial_\lambda u \partial_\mu \partial_\nu u|\leq\frac{C}{1+t}
[|\Gamma u||\partial^2 v|+|\partial u||\partial\Gamma u|+\langle t-r\rangle|\partial u||\partial^2 u|]
\end{equation}
in which $\langle x\rangle=(1+|x|^2)^{\frac{1}{2}}$
\end{lemma}
\begin{proof}
See T. C. Sideris \cite{Sideris_2002}.
\end{proof}
The following Hardy type inequality will be used in the final stage of the energy estimates.
\begin{lemma}\label{Hidano} Suppose that, for every $t>0$, smooth functions $u=u(t,x)$
are compactly supported in $\{x\in \mathbb{R}^3: |x|\leq t+R\}$ with suitable constants $C>0$ and $R>0$. Then the inequality
\begin{equation}
\|\frac{1}{\langle t-r\rangle}u(t)\|_{L^2(\mathbb{R}^3)}\leq C_R\|\nabla u(t)\|_{L^2(\mathbb{R}^3)}
\end{equation}
holds for a constant $C_R$ with $C_R\rightarrow\infty$ as $R\rightarrow\infty$
\end{lemma}
\begin{proof}
See H. Lindblad \cite{Lindblad_1990}.
\end{proof}
In preparation for the energy estimates, we need to consider the commutation
properties of the vector fields Γ with respect to the nonlinear terms.
It is necessary to verify that the null structure is preserved upon differentiation.
\begin{lemma}\label{Sideris1} Let $u$ be solution of (\ref{Equation}).
Assume that the null condition (\ref{Null}) holds for the nonlinearity in (\ref{Nonlinearity}). Then for $|a|\leq l-1$,
\begin{equation}
\Box \Gamma^a u(t)=\sum_{b+c+d=a} N_d(\Gamma^b u, \Gamma^c u),
\end{equation}
in which each $N_d$ is a quadratic nonlinearity of the form (\ref{Nonlinearity})
satisfying (\ref{Null}). Moreover, if $b+c=a$, then $N_d=N.$
\end{lemma}
\begin{proof}
See L. H\"{o}rmander \cite{Hormander_1997} (see also T. C. Sideris \cite{Sideris_2002}).
\end{proof}

\section{Energy estimates}
In this section, we will derive a priori nonlinear energy estimates for
the system (\ref{Equation}). To complete the proof of theorem \ref{Main},
let us first say a few words on the plan of our energy integral argument.
For initial data $(\varphi, \psi) \in H^s_0(B_R)\times H^{s-1}_0(B_R), s \geq 5$, let us
assume $E_s(u(t))<2\epsilon$ for all $0\leq t<T_0$.
Then we will show that $T_0 = \infty$.
We use the method by contradiction. Suppose $T_0 < \infty$ is the largest
time so that $E_s \leq 2\epsilon$. We are going to show that $E_s < 2\epsilon$
for all $0 \leq t \leq T_0$. This strict inequality implies that $T_0$ is not
the largest time, which contradicts to the definition of $T_0$.
This contradiction implies that $T_0 = \infty$. Suppose $0\leq t<T_0$ in
what follows. Denote by $\langle\cdot, \cdot\rangle$ the scalar product
in $\mathbb{R}^3$. Following the energy method, we have for each $s\geq 5$,
\begin{equation}
\begin{split}
E'_s(u(t))&=\sum_{|a|\leq s-1} \int_{\mathbb{R}^3}\langle\Box \Gamma^a u(t), \partial_t\Gamma^a u(t)e^{-q}\rangle dx\\
&\ \ \ \ -\sum_{|a|\leq s-1}(\int_{\mathbb{R}^3} \nabla \Gamma^a
u\partial_t\Gamma^a u e^{-q}q'\frac{x}{r}dx+\frac{1}{2}\int_{\mathbb{R}^3} |\partial_t\Gamma^a u|^2 e^{-q}q'dx\\
&\ \ \ \ +\frac{1}{2}\int_{\mathbb{R}^3} |\nabla\Gamma^a u|^2 e^{-q}q'dx)\\
&=\sum_{|a|\leq s-1} \int_{\mathbb{R}^3}\langle\Box \Gamma^a u(t), \partial_t\Gamma^a u(t)e^{-q}\rangle dx\\
&\ \ \ \ -\sum_{|a|\leq s-1}\frac{1}{2}\int_{\mathbb{R}^3} |\partial_t\Gamma^a u\frac{x}{r}+\nabla \Gamma^a u|^2 e^{-q}q'dx,\\
\end{split}
\end{equation}
define
$$G=\sum_{|a|\leq s-1}\frac{1}{2}\int_{\mathbb{R}^3} |\partial_t\Gamma^a u\frac{x}{r}+\nabla \Gamma^a u|^2 e^{-q}q'dx$$
and from Lemma \ref{Sideris1}, we obtain
\begin{equation}
\begin{split}
E'_s(u(t))+G&=\sum_{|a|\leq s-1}\sum_{b+c+d=a} \int_{\mathbb{R}^3}\langle N_d(\Gamma^b u, \Gamma^c u), \partial_t\Gamma^a u(t)e^{-q}\rangle dx\\
&=\sum_{|a|\leq s-1}\sum_{\substack{b+c+d=a\\d=0}} \int_{\mathbb{R}^3}\langle N_d(\Gamma^b u, \Gamma^c u), \partial_t\Gamma^a u(t)e^{-q}\rangle dx\\
&\ \ \ \ +\sum_{|a|\leq s-1}\sum_{\substack{b+c+d=a\\d\neq0}} \int_{\mathbb{R}^3}\langle N_d(\Gamma^b u, \Gamma^c u)), \partial_t\Gamma^a u(t)e^{-q}\rangle dx.\\
\end{split}
\label{E}
\end{equation}
The first terms above with $b=0$, $c=a$, and $|a|=s-1$ is handled with the aid of the symmetry condition
(\ref{symmetric}) which allows us to integrate by parts as follows.
Recalling that from Lemma \ref{Sideris1}, $N_d=N$ when $b+c=a$.
\begin{equation}
\begin{split}
&\ \ \ \ \int_{\mathbb{R}^3}\langle N(u,\Gamma^a u), \partial_t\Gamma^a u\cdot e^{-q}\rangle dx
=B_{\lambda\mu\nu}\int_{\mathbb{R}^3} \partial_\lambda u\partial_\mu\partial_\nu \Gamma^a u\partial_t\Gamma^a u\cdot e^{-q}dx\\
&=B_{\lambda\mu\nu}\int_{\mathbb{R}^3} \partial_\nu(\partial_\lambda u\partial_\mu \Gamma^a u\partial_t\Gamma^a u\cdot e^{-q})dx
-B_{\lambda\mu\nu}\int_{\mathbb{R}^3} \partial_\nu \partial_\lambda u\partial_\mu \Gamma^a u\partial_t\Gamma^a u\cdot e^{-q}dx\\
&\ \ \ \ -B_{\lambda\mu\nu}\int_{\mathbb{R}^3} \partial_\lambda u\partial_\mu \Gamma^au\partial_\nu\partial_t\Gamma^a u\cdot e^{-q}dx
-B_{\lambda\mu\nu}\int_{\mathbb{R}^3} \partial_\lambda u\partial_\mu \Gamma^au\partial_t\Gamma^a u\cdot \partial_\nu e^{-q}dx\\
&=B_{\lambda\mu0}\partial_t\int_{\mathbb{R}^3} \partial_\lambda u\partial_\mu \Gamma^a u\partial_t\Gamma^a u\cdot e^{-q}dx
-B_{\lambda\mu\nu}\int_{\mathbb{R}^3} \partial_\nu \partial_\lambda u\partial_\mu \Gamma^a u\partial_t\Gamma^a u\cdot e^{-q}dx\\
&\ \ \ \ -\frac{1}{2}B_{\lambda\mu\nu}\int_{\mathbb{R}^3} \partial_\lambda u\partial_t(\partial_\mu \Gamma^au\partial_\nu\Gamma^a u)\cdot e^{-q}dx
-B_{\lambda\mu\nu}\int_{\mathbb{R}^3} \partial_\lambda u\partial_\mu \Gamma^au\partial_t\Gamma^a u\cdot \partial_\nu e^{-q}dx\\
&=\frac{1}{2}B_{\lambda\mu\nu}\eta_{\nu\delta}\partial_t\int_{\mathbb{R}^3} \partial_\lambda u\partial_\mu \Gamma^a u\partial_\delta\Gamma^a u\cdot e^{-q}dx
-B_{\lambda\mu\nu}\int_{\mathbb{R}^3} \partial_\nu \partial_\lambda u\partial_\mu \Gamma^a u\partial_t\Gamma^a u\cdot e^{-q}dx\\
&\ \ \ \ +\frac{1}{2}B_{\lambda\mu\nu}\int_{\mathbb{R}^3} \partial_t\partial_\lambda u\partial_\mu \Gamma^au\partial_\nu\Gamma^a u\cdot e^{-q}dx
+\frac{1}{2}B_{\lambda\mu\nu}\int_{\mathbb{R}^3} \partial_\lambda u\partial_\mu \Gamma^au\partial_\nu\Gamma^a u\cdot e^{-q}(-q')dx\\
&\ \ \ \ -B_{\lambda\mu\nu}\int_{\mathbb{R}^3} \partial_\lambda u\partial_\mu \Gamma^au\partial_t\Gamma^a u\cdot \partial_\nu e^{-q}dx,
\end{split}
\label{E_2}
\end{equation}
using the symbol $\eta_{\gamma\delta}=\text{Diag}[1, -1, -1, -1]$.
The first term above can be absorbed into the energy as a lower
order perturbation. Define
\begin{equation}
\widetilde{E}_s(u(t))=E_s(u(t))-\frac{1}{2}B_{\lambda\mu\nu}\eta_{\nu\delta}
\sum_{|a|=s-1}\int_{\mathbb{R}^3} \partial_\lambda u\partial_\mu \Gamma^au\partial_\delta\Gamma^a u\cdot e^{-q}dx.
\end{equation}
The perturbation is bounded by $C\|\nabla u\|_{L^{\infty}}E_s(u(t))$,
but by Sobolev Imbedding theorem, the maximum norm $\|\nabla u\|_{L^{\infty}}$
is controlled by $E_3^{\frac{1}{2}}\leq E_s^{\frac{1}{2}}<2\epsilon$. Thus, for small solutions we have
\begin{equation}
\frac{1}{2}{E}_s(u(t))\leq\widetilde{E}_s(u(t))\leq 2{E}_s(u(t)).
\label{E_P}
\end{equation}
For the last two terms of (\ref{E_2}) we have
\begin{equation}
\begin{split}
&\ \ \ \ \frac{1}{2}B_{\lambda\mu\nu}\int_{\mathbb{R}^3} \partial_\lambda u\partial_\mu \Gamma^au\partial_\nu\Gamma^a u\cdot e^{-q}(-q')dx\\
&\ \ \ \ -B_{\lambda\mu\nu}\int_{\mathbb{R}^3} \partial_\lambda u\partial_\mu \Gamma^au\partial_t\Gamma^a u\cdot \partial_\nu e^{-q}dx\\
&=\frac{1}{2}B_{\lambda\mu\nu}\int_{\mathbb{R}^3} \partial_\lambda u\partial_\mu \Gamma^au\partial_\nu\Gamma^a u\cdot e^{-q}q'dx\\
&\ \ \ \ -B_{\lambda\mu i}\int_{\mathbb{R}^3} \partial_\lambda u\partial_\mu \Gamma^au(\partial_t\Gamma^a u\frac{x_i}{r}+\partial_i\Gamma^a u)\cdot e^{-q}q'dx.\\
\end{split}
\label{E_3}
\end{equation}
Returning to (\ref{E}), we have derive the following energy identity:
\begin{equation}
\begin{split}
&\ \ \ \ E'_s(u(t))+G-\frac{1}{2}B_{\lambda\mu\nu}\eta_{\nu\delta}
\sum_{|a|=s-1}\int_{\mathbb{R}^3} \partial_\lambda u\partial_\mu \Gamma^au\partial_\delta\Gamma^a u\cdot e^{-q}dx\\
&=\frac{1}{2}B_{\lambda\mu\nu}\sum_{|a|=s-1}\int_{\mathbb{R}^3} \partial_t\partial_\lambda u\partial_\mu \Gamma^au\partial_\nu\Gamma^a u\cdot e^{-q}dx\\
&\ \ \ \ -B_{\lambda\mu\nu}\sum_{|a|=s-1}\int_{\mathbb{R}^3} \partial_\nu \partial_\lambda u\partial_\mu \Gamma^a u\partial_t\Gamma^a u\cdot e^{-q}dx\\
&\ \ \ \ -B_{\lambda\mu i}\sum_{|a|= s-1}\int_{\mathbb{R}^3} \partial_\lambda u\partial_\mu \Gamma^au(\partial_t\Gamma^a u\frac{x_i}{r}+\partial_i\Gamma^a u)\cdot e^{-q}q'dx\\
&\ \ \ \ +\frac{1}{2}B_{\lambda\mu\nu}\sum_{|a|= s-1}\int_{\mathbb{R}^3} \partial_\lambda u\partial_\mu \Gamma^au\partial_\nu\Gamma^a u\cdot e^{-q}q'dx\\
&\ \ \ \ +\sum_{|a|= s-1}\sum_{\substack{b+c=a\\c\neq a}} \int_{\mathbb{R}^3}\langle N(\Gamma^b u, \Gamma^c u), \partial_t\Gamma^a u(t)e^{-q}\rangle dx\\
&\ \ \ \ +\sum_{|a|\leq s-2}\sum_{b+c=a} \int_{\mathbb{R}^3}\langle N(\Gamma^b u, \Gamma^c u)), \partial_t\Gamma^a u(t)e^{-q}\rangle dx\\
&\ \ \ \ +\sum_{|a|\leq s-1}\sum_{\substack{b+c+d=a\\d\neq0}} \int_{\mathbb{R}^3}\langle N_d(\Gamma^b u, \Gamma^c u), \partial_t\Gamma^a u(t)e^{-q}\rangle dx\\
&:=J_1+J_2+J_3+J_4+J_5+J_6+J_7.
\label{E_4}
\end{split}
\end{equation}
In what follows, we shall estimate each terms on the right hand side of (\ref{E_4}).\\
{\bf Estimates for $J_3$}.\\
By H\"{o}lder's inequality and Lemma \ref{lemma3}, the terms $J_3$ can be bounded as follows:
\begin{equation}
\begin{split}
J_3&\leq C\int_{\mathbb{R}^3}|\partial u|_{L^\infty}|\partial\Gamma^a ue^{-\frac{q}{2}}||(\partial_t\Gamma^a u\frac{x_i}{r}+\partial_i\Gamma^a u)\cdot e^{-\frac{q}{2}}q'|dx\\
&\leq\frac{C}{1+t}E_sG^{\frac{1}{2}}\leq\frac{C}{(1+t)^2}E_s^2+\epsilon G.
\label{E_6}
\end{split}
\end{equation}
To estimate the other terms on the right-hand side of (\ref{E_4}). We will divide the integral region $\mathbb{R}^3$ into
two parts: {\itshape{inside the cones}} $\{(t, x): |x|\leq\frac{t+1}{2}\}$
and {\itshape{away from the origin}} $\{(t, x): |x|\geq\frac{t+1}{2}\}$\\
{\bf Estimates for other terms inside the cones:  $\{(t, x): |x|\leq\frac{t+1}{2}\}$}.\\
\noindent {\itshape{Inside the cones}}. On the region:  $|x|\leq\frac{t+1}{2}$, the terms $J_1, J_2, J_5, J_6, J_7$ can be bounded by
\begin{equation}
\begin{split}
\sum_{|a|\leq s-1}\int_{\mathbb{R}^3} |\partial^2\Gamma^bu\partial\Gamma^cu\partial\Gamma^aue^{-q}|dx
\end{split}
\label{E_5}
\end{equation}
with $|b+c|\leq s-1$ and $|b|\leq s-2$.
Noticing that $s\geq 5$, combining with Lemma \ref{lemma3}, we obtain
\begin{equation}
(\ref{E_5})\leq \frac{C}{(1+t)^{\frac{3}{2}}}E^{\frac{3}{2}}_s(u(t)).
\label{E_7}
\end{equation}
For the terms $J_4$, by Lemma \ref{lemma3} again, we have
\begin{equation}
\begin{split}
J_4&\leq C\sum_{|a|= s-1}\int_{\mathbb{R}^3} \partial u\partial\Gamma^au\partial\Gamma^aue^{-q}q'dx\\
&\leq C\sum_{|a|= s-1}\|\partial u\|_{L\infty}\|\partial\Gamma^aue^{-q/2}\|^2_{L^2}\\
&\leq \frac{C}{(1+t)^{\frac{3}{2}}}E^{\frac{3}{2}}_s(u(t)).\\
\end{split}
\label{E_8}
\end{equation}
Combining the estimates (\ref{E_6})(\ref{E_7})(\ref{E_8}) and
taking (\ref{E_P}) into account, we obtain the following energy inequality:
\begin{equation}
\begin{split}
\widetilde{E}'_s(u(t))+G\leq\frac{C}{(1+t)^2}
\widetilde{E}^2_s(u(t))+\frac{C}{(1+t)^{\frac{3}{2}}}\widetilde{E}^{\frac{3}{2}}_s(u(t))+\epsilon G,
\end{split}
\end{equation}
for the portion of the integrals over $|x|\leq\frac{t+1}{2}$.\\
{\itshape{Away from the origin}}. On the region:  $|x|\geq\frac{t+1}{2}$,
we may first focus on the estimates of $J_1$ and $J_2$.
As the proof of $J_1$ and $J_2$ is similar, we just present the proof of $J_1$. We leave the proof of $J_2$
to interested readers.\\
{\bf Estimates for $J_1$}.\\
To bound the terms $J_1$, we sort them out into the following four case:\\
{\bf Case a:} all of $\lambda$, $\mu,$ $\nu$ are not zero, $\lambda$,
$\mu,$ $\nu$ $\in\{1, 2, 3\}$. In order to obtain the uniform Sobolev $H^s$ norms,
the terms $G$ play an important role. First, we split $J_1$ as follows:
\begin{equation}
\begin{split}
&\ \ \ \ \frac{1}{2}B_{\lambda\mu\nu}\sum_{|a|= s-1}\int_{\mathbb{R}^3} \partial_t
\partial_\lambda u\partial_\mu \Gamma^au\partial_\nu\Gamma^a u\cdot e^{-q}dx\\
&=\frac{1}{2}B_{\lambda\mu\nu}\sum_{|a|= s-1}\int_{\mathbb{R}^3} \partial_t\partial_\lambda u(\partial_\mu \Gamma^au+\partial_t \Gamma^au\frac{x_\mu}{r})\partial_\nu\Gamma^a u\cdot e^{-q}dx\\
&\ \ \ \ -\frac{1}{2}B_{\lambda\mu\nu}\sum_{|a|=s-1}\int_{\mathbb{R}^3} \partial_t
\partial_\lambda u\partial_t \Gamma^au\frac{x_\mu}{r}
(\partial_\nu\Gamma^a u+\partial_t \Gamma^au\frac{x_\nu}{r}) e^{-q}dx\\
&\ \ \ \ +\frac{1}{2}B_{\lambda\mu\nu}\sum_{|a|=s-1}\int_{\mathbb{R}^3}
\partial_t\partial_\lambda u\frac{x_\mu}{r}\frac{x_\nu}{r}\partial_t \Gamma^au\partial_t \Gamma^au e^{-q}dx\\
&:=J_{11}+J_{12}+J_{13}.
\end{split}
\end{equation}
We need estimate the above terms one by one. First, in light of Lemma \ref{lemma1}, Lemma \ref{lemma3}:
\begin{equation}
\begin{split}
J_{11}&=\frac{1}{2}B_{\lambda\mu\nu}\sum_{|a|= s-1}\int_{\mathbb{R}^3} \partial_t\partial_\lambda u\frac{1}{\sqrt{q'}}(\partial_\mu \Gamma^au+\partial_t
\Gamma^au\frac{x_\mu}{r})\sqrt{q'}\partial_\nu\Gamma^a u\cdot e^{-q}dx\\
&\leq\frac{1}{2}B_{\lambda\mu\nu}\sum_{|a|=s-1}\int_{\mathbb{R}^3} \partial_t\Gamma u(\partial_\mu \Gamma^au+\partial_t \Gamma^au\frac{x_\mu}{r})\sqrt{q'}\partial_\nu\Gamma^a u\cdot e^{-q}dx\\
&\leq \frac{C}{1+t}E_sG^{\frac{1}{2}}\leq\frac{C}{(1+t)^2}E_s^2+\epsilon G.
\end{split}
\end{equation}
Similarly
\begin{equation}
\begin{split}
J_{12}\leq \frac{C}{1+t}E_sG^{\frac{1}{2}}\leq\frac{C}{(1+t)^2}E_s^2+\epsilon G.
\end{split}
\end{equation}
In order to use the null condition, using the spatial derivatives decompose (\ref{Decompose}), we split the terms $J_{13}$ as follows:
\begin{equation}
\begin{split}
J_{13}&=\frac{1}{2}B_{\lambda\mu\nu}\sum_{|a|= s-1}\int_{\mathbb{R}^3} \partial_t\partial_\lambda u\frac{x_\mu}{r}\frac{x_\nu}{r}\partial_t \Gamma^au\partial_t \Gamma^au e^{-q}dx\\
&=\frac{1}{2}B_{\lambda\mu\nu}\sum_{|a|=s-1}\int_{\mathbb{R}^3} \partial_t\partial_r u\frac{x_\lambda}{r}\frac{x_\mu}{r}\frac{x_\nu}{r}\partial_t \Gamma^au\partial_t \Gamma^au e^{-q}dx\\
&\ \ \ \ -\frac{1}{2}B_{\lambda\mu\nu}\sum_{|a|= s-1}\int_{\mathbb{R}^3} \partial_t(\frac{x}{r^2}\wedge \Omega)_\lambda u\frac{x_\mu}{r}\frac{x_\nu}{r}\partial_t \Gamma^au\partial_t \Gamma^au e^{-q}dx\\
&:=J_{131}+J_{132}.
\end{split}
\end{equation}
Using Lemma \ref{lemma3}, the second terms are bounded by
\begin{equation}
\begin{split}
J_{132}&\leq\frac{C}{1+t}\sum_{|a|= s-1}\int_{\mathbb{R}^3} \partial\Gamma u\partial \Gamma^au\partial \Gamma^au e^{-q}dx\\
&\leq\frac{C}{(1+t)^2}E_s^{\frac{3}{2}},
\end{split}
\end{equation}
for $|x|\geq\frac{t+1}{2}$.\\
{\bf{case b:}} Only one of $\lambda,\mu,\nu$ is zero.\\
1)For $\lambda=0$, $\mu$, $\nu$ $\in\{1, 2, 3\}$, we have
\begin{equation}
\begin{split}
&\ \ \ \ \frac{1}{2}B_{0\mu\nu}\sum_{|a|= s-1}\int_{\mathbb{R}^3} \partial_t\partial_t u\partial_\mu \Gamma^au\partial_\nu\Gamma^a u\cdot e^{-q}dx\\
&=\frac{1}{2}B_{0\mu\nu}\sum_{|a|= s-1}\int_{\mathbb{R}^3} \partial_t(\partial_t+\partial_r) u\partial_\mu \Gamma^au\partial_\nu\Gamma^a u\cdot e^{-q}dx\\
&\ \ \ \ -\frac{1}{2}B_{0\mu\nu}\sum_{|a|=s-1}\int_{\mathbb{R}^3} \partial_t\partial_r u\partial_\mu \Gamma^au\partial_\nu\Gamma^a u\cdot e^{-q}dx\\
&:=K_1+K_2.
\end{split}
\end{equation}
By Lemma \ref{lemma4}, Lemma \ref{lemma3}, the first terms $K_1$ can be bounded as follows:
\begin{equation}
\begin{split}
K_1&\leq\frac{1}{2}B_{0\mu\nu}\sum_{|a|= s-1}\int_{\mathbb{R}^3} \frac{\partial_t\Gamma u}{1+t}\partial_\mu \Gamma^au\partial_\nu\Gamma^a u\cdot e^{-q}dx\\
&\leq\frac{C}{(1+t)^2}E_s^{\frac{3}{2}}.
\end{split}
\end{equation}
On the other hand, for the second terms $K_2$, we split it as follows:
\begin{equation}
\begin{split}
K_2&=-\frac{1}{2}B_{0\mu\nu}\sum_{|a|= s-1}\int_{\mathbb{R}^3} \partial_t\partial_r u\partial_\mu \Gamma^au\partial_\nu\Gamma^a u\cdot e^{-q}dx\\
&=-\frac{1}{2}B_{0\mu\nu}\sum_{|a|= s-1}\int_{\mathbb{R}^3} \partial_t\partial_r u(\partial_\mu \Gamma^au+\partial_t \Gamma^au\frac{x_\mu}{r})\partial_\nu\Gamma^a u\cdot e^{-q}dx\\
&\ \ \ \ +\frac{1}{2}B_{0\mu\nu}\sum_{|a|= s-1}\int_{\mathbb{R}^3} \partial_t\partial_r u\partial_t \Gamma^au\frac{x_\mu}{r}(\partial_\nu\Gamma^a u+\partial_t \Gamma^au\frac{x_\lambda}{r})\cdot e^{-q}dx\\
&\ \ \ \ -\frac{1}{2}B_{0\mu\nu}\sum_{|a|= s-1}\int_{\mathbb{R}^3} \partial_t\partial_r u\partial_t \Gamma^au\frac{x_\mu}{r}\partial_t \Gamma^au\frac{x_\nu}{r}\cdot e^{-q}dx\\
&:=K_{21}+K_{22}+K_{23}.
\end{split}
\end{equation}
The terms $K_{21}, K_{22}$ are handled in a similar fashion as $J_{11}$.\\
\begin{equation}
\begin{split}
K_{21}+K_{22}\leq \frac{C}{1+t}E_sG^{\frac{1}{2}}\leq\frac{C}{(1+t)^2}E_s^2+\epsilon G.
\end{split}
\end{equation}
2)Similarly, for $\mu=0$, $\lambda,$ $\nu$ $\in\{1, 2, 3\}$, we proceed as before:
\begin{equation}
\begin{split}
&\ \ \ \ \frac{1}{2}B_{\lambda0\nu}\sum_{|a|= s-1}\int_{\mathbb{R}^3} \partial_t\partial_\lambda u\partial_\nu \Gamma^au\partial_t\Gamma^a u\cdot e^{-q}dx\\
&=\frac{1}{2}B_{\lambda0\nu}\sum_{|a|= s-1}\int_{\mathbb{R}^3} \partial_t\partial_\lambda u(\partial_\nu \Gamma^au+\partial_t \Gamma^au\frac{x_\nu}{r})\partial_t\Gamma^a u\cdot e^{-q}dx\\
&\ \ \ \ -\frac{1}{2}B_{\lambda0\nu}\sum_{|a|= s-1}\int_{\mathbb{R}^3} \partial_t\partial_r u\partial_t \Gamma^au\frac{x_\lambda}{r}\frac{x_\nu}{r}\partial_t\Gamma^a u\cdot e^{-q}dx\\
&\ \ \ \ +\frac{1}{2}B_{\lambda0\nu}\sum_{|a|= s-1}\int_{\mathbb{R}^3} \partial_t(\frac{x}{r^2}\wedge \Omega)_\lambda u\partial_t \Gamma^au\frac{x_\nu}{r}\partial_t\Gamma^a u\cdot e^{-q}dx\\
&:=L_1+L_2+L_3.
\end{split}
\end{equation}
3)For $\nu=0$, $\lambda$, $\mu$ $\in\{1, 2, 3\}$, we have
\begin{equation}
\begin{split}
&\ \ \ \ \frac{1}{2}B_{\lambda\mu0}\sum_{|a|= s-1}\int_{\mathbb{R}^3} \partial_t\partial_\lambda u\partial_t \Gamma^au\partial_\mu\Gamma^a u\cdot e^{-q}dx\\
&=\frac{1}{2}B_{\lambda\mu0}\sum_{|a|= s-1}\int_{\mathbb{R}^3} \partial_t\partial_\lambda u\partial_t\Gamma^a u(\partial_\mu \Gamma^au+\partial_t \Gamma^au\frac{x_\mu}{r})\cdot e^{-q}dx\\
&\ \ \ \ -\frac{1}{2}B_{\lambda\mu0}\sum_{|a|=s-1}\int_{\mathbb{R}^3} \partial_t\partial_r u\partial_t \Gamma^au\frac{x_\lambda}{r}\frac{x_\mu}{r}\partial_t\Gamma^a u\cdot e^{-q}dx\\
&\ \ \ \ +\frac{1}{2}B_{\lambda\mu0}\sum_{|a|= s-1}\int_{\mathbb{R}^3} \partial_t(\frac{x}{r^2}\wedge \Omega)_\lambda u\partial_t \Gamma^au\frac{x_\mu}{r}\partial_t\Gamma^a u\cdot e^{-q}dx\\
&:=M_1+M_2+M_3.
\end{split}
\end{equation}
For the estimate of the terms $L_1$, $M_1$, we proceed as in $J_{11}$ to obtain
\begin{equation}
\begin{split}
L_1+M_1\leq \frac{C}{1+t}E_sG^{\frac{1}{2}}\leq\frac{C}{(1+t)^2}E_s^2+\epsilon G.
\end{split}
\end{equation}
The terms $L_3, M_3$ are handled in a similar fashion as $J_{132}$\\
\begin{equation}
\begin{split}
L_3+M_3\leq\frac{C}{(1+t)^2}E_s^{\frac{3}{2}}.
\end{split}
\end{equation}
{\bf Case c:} two of $\lambda$, $\mu,$ $\nu$ are zero.\\
1)For $\lambda=\mu=0$, $\nu$ $\in\{1, 2, 3\}$, we have
\begin{equation}
\begin{split}
&\ \ \ \ \frac{1}{2}B_{00\nu}\sum_{|a|= s-1}\int_{\mathbb{R}^3} \partial_t\partial_t u\partial_\nu \Gamma^au\partial_t\Gamma^a u\cdot e^{-q}dx\\
&=\frac{1}{2}B_{00\nu}\sum_{|a|= s-1}\int_{\mathbb{R}^3} \partial_t\partial_t u(\partial_\nu \Gamma^au+\partial_t \Gamma^au\frac{x_\nu}{r})\partial_t\Gamma^a u\cdot e^{-q}dx\\
&\ \ \ \ -\frac{1}{2}B_{00\nu}\sum_{|a|= s-1}\int_{\mathbb{R}^3} \partial_t(\partial_t+\partial_r) u\partial_t \Gamma^au\frac{x_\nu}{r}\partial_t\Gamma^a u\cdot e^{-q}dx\\
&\ \ \ \ +\frac{1}{2}B_{00\nu}\sum_{|a|= s-1}\int_{\mathbb{R}^3} \partial_t\partial_r u\partial_t \Gamma^au\frac{x_\nu}{r}\partial_t\Gamma^a u\cdot e^{-q}dx\\
&:=N_{11}+N_{12}+N_{13}.
\end{split}
\end{equation}
2)For $\mu=\nu=0$, $\lambda$ $\in\{1, 2, 3\}$, we have
\begin{equation}
\begin{split}
&\ \ \ \ \frac{1}{2}B_{\lambda00}\sum_{|a|= s-1}\int_{\mathbb{R}^3} \partial_t\partial_\lambda u\partial_t \Gamma^au\partial_t\Gamma^a u\cdot e^{-q}dx\\
&=\frac{1}{2}B_{\lambda00}\sum_{|a|= s-1}\int_{\mathbb{R}^3} \partial_t\partial_r u\frac{x_\lambda}{r}\partial_t \Gamma^au\partial_t\Gamma^a u\cdot e^{-q}dx\\
&\ \ \ \ -\frac{1}{2}B_{\lambda00}\sum_{|a|= s-1}\int_{\mathbb{R}^3} \partial_t(\frac{x}{r^2}\wedge \Omega)_\lambda u\partial_t \Gamma^au\partial_t\Gamma^a u\cdot e^{-q}dx\\
&:=N_{21}+N_{22}.
\end{split}
\end{equation}
3)For $\lambda=\nu=0$, $\lambda$ $\in\{1, 2, 3\}$, we have
\begin{equation}
\begin{split}
&\ \ \ \ \frac{1}{2}B_{0\mu0}\sum_{|a|= s-1}\int_{\mathbb{R}^3} \partial_t\partial_t u\partial_\mu \Gamma^au\partial_t\Gamma^a u\cdot e^{-q}dx\\
&=\frac{1}{2}B_{0\mu0}\sum_{|a|= s-1}\int_{\mathbb{R}^3} \partial_t\partial_t u(\partial_\mu \Gamma^au+\partial_t \Gamma^au\frac{x_\mu}{r})\partial_t\Gamma^a u\cdot e^{-q}dx\\
&\ \ \ \ -\frac{1}{2}B_{0\mu0}\sum_{|a|= s-1}\int_{\mathbb{R}^3} \partial_t(\partial_t+\partial_r) u\partial_t \Gamma^au\frac{x_\mu}{r}\partial_t\Gamma^a u\cdot e^{-q}dx\\
&\ \ \ \ +\frac{1}{2}B_{0\mu0}\sum_{|a|= s-1}\int_{\mathbb{R}^3} \partial_t\partial_r u\partial_t \Gamma^au\frac{x_\mu}{r}\partial_t\Gamma^a u\cdot e^{-q}dx\\
&:=N_{31}+N_{32}+N_{33}.
\end{split}
\end{equation}
For the estimate of the terms $N_{11}, N_{31}$, we proceed as in $J_{11}$ to obtain:
\begin{equation}
\begin{split}
N_{11}+ N_{31}\leq \frac{C}{1+t}E_sG^{\frac{1}{2}}\leq\frac{C}{(1+t)^2}E_s^2+\epsilon G.
\end{split}
\end{equation}
The terms $N_{22}$ are handled in a similar fashion as $J_{132}$\\
\begin{equation}
\begin{split}
N_{22}\leq\frac{C}{(1+t)^2}E_s^{\frac{3}{2}}.
\end{split}
\end{equation}
The terms $N_{12}, N_{32}$ are handled in a similar fashion as $K_1$. By Lemma \ref{lemma4}, Lemma \ref{lemma3}
\begin{equation}
N_{12}+N_{32}\leq\frac{C}{(1+t)^2}E^{\frac{3}{2}}_s(u(t)).
\end{equation}
{\bf Case d:} For $\lambda=\mu=\nu=0$, we have
\begin{equation}
\begin{split}
&\ \ \ \ \frac{1}{2}B_{000}\sum_{|a|= s-1}\int_{\mathbb{R}^3} \partial_t\partial_t u\partial_t \Gamma^au\partial_t\Gamma^a u\cdot e^{-q}dx\\
&=\frac{1}{2}B_{000}\sum_{|a|= s-1}\int_{\mathbb{R}^3} \partial_t(\partial_t+\partial_r) u\partial_t \Gamma^au\partial_t\Gamma^a u\cdot e^{-q}dx\\
&\ \ \ \ -\frac{1}{2}B_{000}\sum_{|a|= s-1}\int_{\mathbb{R}^3} \partial_t\partial_r u\partial_t \Gamma^au\partial_t\Gamma^a u\cdot e^{-q}dx\\
&=O_1+O_2.
\end{split}
\end{equation}
The terms $O_1$ are handled in a similar fashion as $K_1$. By Lemma \ref{lemma4}, Lemma \ref{lemma3}
\begin{equation}
O_1\leq\frac{C}{(1+t)^2}E^{\frac{3}{2}}_s(u(t)).
\end{equation}
Finally, noticing that $x_0=-1$, the terms $J_{131}, K_{23}, L_2, M_2, N_{13}, N_{21}, N_{33}, O_{2}$ vanish since the null condition.
Therefore, we deduce from the four cases above that
\begin{equation}
\begin{split}
J_1\leq\frac{C}{(1+t)^2}E^2_s(u(t))+\frac{C}{(1+t)^2}E^{\frac{3}{2}}_s(u(t))+\frac{C}{(1+t)^{\frac{3}{2}}}E^{\frac{3}{2}}_s(u(t))+\epsilon G,
\end{split}
\end{equation}
for the portion of the integrals over $|x|\geq\frac{t+1}{2}$.\\
{\bf Estimates for $J_4$}.\\
$J_4$ is treated in a similar (in fact, easier) way as $J_1$, so we omit the detail, here we just point out the difference. For example, for $\lambda=\mu=\nu=0$, we have
\begin{equation}
\begin{split}
&\ \ \ \ \frac{1}{2}B_{000}\sum_{|a|= s-1}\int_{\mathbb{R}^3} \partial_t u\partial_t \Gamma^au\partial_t\Gamma^a u\cdot e^{-q}q'dx\\
&=\frac{1}{2}B_{000}\sum_{|a|= s-1}\int_{\mathbb{R}^3} (\partial_t+\partial_r) u\partial_t \Gamma^au\partial_t\Gamma^a u\cdot e^{-q}q'dx\\
&\ \ \ \ -\frac{1}{2}B_{000}\sum_{|a|= s-1}\int_{\mathbb{R}^3} \partial_r u\partial_t \Gamma^au\partial_t\Gamma^a u\cdot e^{-q}q'dx\\
&:=J_{41}+J_{42}.
\end{split}
\end{equation}
Using Lemma \ref{lemma4}, we obtain
\begin{equation}
\begin{split}
J_{41}\leq\int_{\mathbb{R}^3} \frac{\Gamma u}{1+t}\partial \Gamma^au\partial\Gamma^a u\cdot e^{-q}q'dx.
\end{split}
\end{equation}
Using Lemma \ref{lemma2}, we obtain
\begin{equation}
\begin{split}
J_{41}&\leq \frac{C}{(1+t)^{\frac{1}{2}}}\int_{\mathbb{R}^3} \frac{r^{\frac{1}{2}}\Gamma u}{1+t}\partial \Gamma^au\partial\Gamma^a u\cdot e^{-q}q'dx\\
&\leq\frac{C}{(1+t)^{\frac{3}{2}}}E_s^{\frac{3}{2}}.
\end{split}
\end{equation}
for the portion of the integrals over $r\geq\frac{t+1}{2}$.
This is the crucial point that helps us obtain the energy. Therefore, we deduce that
\begin{equation}
\begin{split}
J_4\leq\frac{C}{(1+t)^2}E^2_s(u(t))+\frac{C}{(1+t)^{\frac{3}{2}}}E^{\frac{3}{2}}_s(u(t))+\epsilon G.
\end{split}
\end{equation}
{\bf Estimates for $J_5, J_6, J_7$}. \\
To complete the proof of Theorem \ref{Main},
it remains to deal with the terms $J_5$, $J_6$, $J_7$, which could be estimated of the form:
\begin{equation}
\begin{split}
A=\sum_{|a|\leq s-1}\sum_{\substack{b+c=a\\b\neq a, c\neq a}}\int_{\mathbb{R}^3} \partial^2\Gamma^bu\partial \Gamma^c u\partial_t\Gamma^a u\cdot e^{-q}dx.
\end{split}
\end{equation}
We notice that, by Lemma \ref{N-inequality}, Lemma \ref{lemma1}, we have
\begin{equation}
\begin{split}
A&\leq\sum_{|a|\leq s-1}\sum_{\substack{b+c=a\\b\neq a, c\neq a}}\frac{C}{1+t}\int_{\mathbb{R}^3} |\partial^2\Gamma^bu||\Gamma^{c+1} u|\partial\Gamma^a u| e^{-q}dx\\
&\ \ \ \ +\sum_{|a|\leq s-1}\sum_{\substack{b+c=a\\b\neq a, c\neq a}}\frac{C}{1+t}\int_{\mathbb{R}^3} |\partial\Gamma^{b+1}u||\partial\Gamma^c u|\partial\Gamma^a u| e^{-q}dx\\
&:=P_1+P_2.
\end{split}
\end{equation}
In order to estimate the terms $P_2$, notice that $|b+c|\leq s-1$, thus
there is one factor with at most $[\frac{s-1}{2}]$ derivatives. This term
is estimated in $L^\infty$ and the others in $L^2$. The result is that
\begin{equation}
\begin{split}
P_2\leq\frac{C}{(1+t)^{\frac{3}{2}}}E^{\frac{3}{2}}_s(u(t)).
\end{split}
\end{equation}
The terms $P_1$ are much more complicated, here, Lemma \ref{Hidano} play an important role, in order to obtain the energy,
we shall separate the cases in the summation of $P_1$ as follows:
\begin{equation}
\begin{split}
P_1&=\sum_{|a|\leq s-1}\sum_{\substack{b+c=a\\b\neq a, c\neq a}}\frac{C}{1+t}\int_{\mathbb{R}^3} |\partial^2\Gamma^bu||\Gamma^{c+1} u|\partial\Gamma^a u| e^{-q}dx\\
&=\sum_{|a|\leq s-1}\sum_{\substack{b+c=a\\b\neq a, |c|= |a|-1}}\frac{C}{1+t}\int_{\mathbb{R}^3} |\partial^2\Gamma^bu||\Gamma^{c+1} u|\partial\Gamma^a u| e^{-q}dx\\
&\ \ \ \ +\sum_{|a|\leq s-1}\sum_{\substack{b+c=a\\b\neq a, |c|\leq |a|-2}}\frac{C}{1+t}\int_{\mathbb{R}^3} |\partial^2\Gamma^bu||\Gamma^{c+1} u|\partial\Gamma^a u| e^{-q}dx\\
&=P_{11}+P_{12}.
\end{split}
\end{equation}
Using Lemma \ref{Hidano}, Lemma \ref{lemma2}, Lemma \ref{lemma1},
the terms $P_{11}$ are estimated as follows:
\begin{equation}
\begin{split}
P_{11}&=\sum_{|a|\leq s-1}\sum_{\substack{b+c=a\\b\neq a, |c|= |a|-1}}\frac{C}{1+t}\int_{\mathbb{R}^3} \langle t-r\rangle|\partial^2\Gamma^bu|\frac{|\Gamma^{c+1} u|}{\langle t-r\rangle}|\partial\Gamma^a u| e^{-q}dx\\
&\leq\frac{C}{(1+t)^{\frac{3}{2}}} \|r^{\frac{1}{2}}\langle t-r\rangle\partial^2\Gamma^bu\|_{L^\infty}\|\frac{|\Gamma^{c+1} u|}{\langle t-r\rangle}\|_{L^2}\|\partial\Gamma^a u e^{-q/2}\|_{L^2}\\
&\leq \frac{C}{(1+t)^{\frac{3}{2}}}E^{\frac{3}{2}}_s(u(t)).\\
\end{split}
\end{equation}
On the other hand, in order to obtain the energy,
we deal with the terms $P_{12}$ as follows:
\begin{equation}
\begin{split}
P_{12}=\sum_{|a|\leq s-1}\sum_{\substack{b+c=a\\b\neq a, |c|\leq |a|-2}}\frac{C}{1+t}\int_{\mathbb{R}^3} \frac{1}{r^{\frac{1}{2}}}|\partial^2\Gamma^bu|r^{\frac{1}{2}}|\Gamma^{c+1} u||\partial\Gamma^a u| e^{-q}dx.\\
\end{split}
\end{equation}
Using Lemma \ref{lemma2} again, we have
\begin{equation}
\begin{split}
P_{12}\leq \frac{C}{(1+t)^{\frac{3}{2}}}E^{\frac{3}{2}}_s(u(t)).\\
\end{split}
\end{equation}
Collecting the estimates above and taking (\ref{E_P}) into account,
we obtain the following energy inequality:
\begin{equation}
\begin{split}
\widetilde{E}'_s(u(t))+G\leq\frac{C}{(1+t)^2}\widetilde{E}^2_s(u(t))+
\frac{C}{(1+t)^2}\widetilde{E_s}^{\frac{3}{2}}(u(t))+\frac{C}{(1+t)^{\frac{3}{2}}}\widetilde{E}^{\frac{3}{2}}_s(u(t))+\epsilon G,
\end{split}
\end{equation}
for the portion of the integrals over $|x|\geq\frac{t+1}{2}$.
\section{Proof of Theorem \ref{Main}}
Combining all the estimates in section 3, we obtain
\begin{equation}
\begin{split}
\widetilde{E}'_s(u(t))+G\leq\frac{C}{(1+t)^2}\widetilde{E}^2_s(u(t))+
\frac{C}{(1+t)^2}\widetilde{E}^{\frac{3}{2}}_s(u(t))+\frac{C}{(1+t)^{\frac{3}{2}}}\widetilde{E}^{\frac{3}{2}}_s(u(t))+\epsilon G.
\label{LE}
\end{split}
\end{equation}
Since we have assumed that
$$E_{s}(u(0))<\epsilon,$$
it follows from (\ref{E_P}) that
$$\widetilde{E}_s(u(0))<2 \epsilon.$$
Now define
\begin{equation}
T=\sup\{t, \widetilde{E}_s(u(s))<4 \epsilon, 0\leq s\leq t, \epsilon<\frac{1}{4}\},
\end{equation}
then by continuity, $T>0$. From now on, let $0\leq t\leq T$. And from (\ref{LE}), we have
\begin{equation}
\begin{split}
\widetilde{E}'_s(u(t))+(1-\epsilon) G\leq \frac{C}{(1+t)^{\frac{3}{2}}}\widetilde{E}^{\frac{3}{2}}_s(u(t)).
\end{split}
\end{equation}
Thus we have
\begin{equation}
\begin{split}
\widetilde{E}'_s(u(t))\leq \frac{C}{(1+t)^{\frac{3}{2}}}\widetilde{E}^{\frac{3}{2}}_s(u(t))\leq \frac{C}{(1+t)^{\frac{3}{2}}}\widetilde{E}_s(u(t))(4\epsilon)^{\frac{1}{2}}.
\end{split}
\end{equation}
From which, we obtain the following inequality
\begin{equation}
\begin{split}
\widetilde{E}_s(u(t))\leq \widetilde{E}_s(u(0))e^{(4\epsilon)^{\frac{1}{2}}C(1-(1+t)^{-\frac{1}{2}})}.
\end{split}
\end{equation}
Choosing $\epsilon>0$ such that $\epsilon$ satisfies
$$e^{(4\epsilon)^{\frac{1}{2}}C}\leq \frac{3}{2}.$$
Clearly, with such kind of restriction, we obtain
$$\epsilon\leq\frac{1}{4}(\frac{1}{C}\ln\frac{3}{2})^2.$$
Hence, we have
\begin{equation}
\begin{split}
\widetilde{E}_s(u(t))\leq 3\epsilon.
\end{split}
\end{equation}
Thus, $T$ cannot be finite, so $T=\infty$. This completes the proof of Theorem \ref{Main}.

\section*{Acknowledgement}
The author would like to thank Prof. Zhen Lei for direction
and valuable comments on this paper.


\end{document}